\documentclass[11pt,a4paper]{article}
\usepackage[T1]{fontenc}
\usepackage[utf8]{inputenc}
\usepackage[english]{babel}
\usepackage{lmodern,microtype}
\usepackage[a4paper,margin=26mm,headheight=14pt]{geometry}
\usepackage{amsmath,amssymb,amsthm,mathtools}
\usepackage{graphicx,float,caption}
\usepackage{enumitem,xcolor,hyperref}
\definecolor{linkblue}{RGB}{25,65,105}
\hypersetup{colorlinks=true,linkcolor=linkblue,citecolor=linkblue,urlcolor=linkblue,
pdftitle={A shape theorem for a radially excited random walk in dimensions d >= 2}}
\setlist{itemsep=3pt,topsep=5pt}
\numberwithin{equation}{section}
\newtheorem{theorem}{Theorem}[section]
\newtheorem{proposition}[theorem]{Proposition}
\newtheorem{lemma}[theorem]{Lemma}

\theoremstyle{definition}

\theoremstyle{remark}

\newcommand{\R}{\mathbb R}
\newcommand{\Z}{\mathbb Z}
\newcommand{\N}{\mathbb N}
\newcommand{\E}{\mathbb E}
\newcommand{\Pp}{\mathbb P}
\newcommand{\F}{\mathcal F}
\newcommand{\one}{\mathbf 1}
\newcommand{\eps}{\varepsilon}
\newcommand{\dd}{\,\mathrm d}
\newcommand{\cE}{\mathcal E}
\newcommand{\bn}{\mathbf n}
\DeclareMathOperator{\vol}{vol}

\title{A shape theorem for a radially excited random walk in dimensions $d\ge2$}
\author{Arvind Singh}
\date{October 7, 2026}
\begin{document}
\maketitle
\begin{abstract}
We consider a once-excited random walk on $\Z^d$, $d\ge2$, where the walk on its first visit has a bias of constant strength \(\beta>0\) toward the origin and moves like a simple symmetric random walk on subsequent visits to that site. We show that the walk is recurrent and prove an almost sure spherical shape theorem: the trace after \(n\) steps is asymptotically a Euclidean ball centered at the origin, with radius proportional to \(n^{1/(d+1)}\), and the local times have an asymptotically conical profile.
This result confirms a conjecture of Kozma (2007) on the shape of the trace for this model.
\end{abstract}

\begin{figure}[H]
\centering
\includegraphics[width=0.45\textwidth]{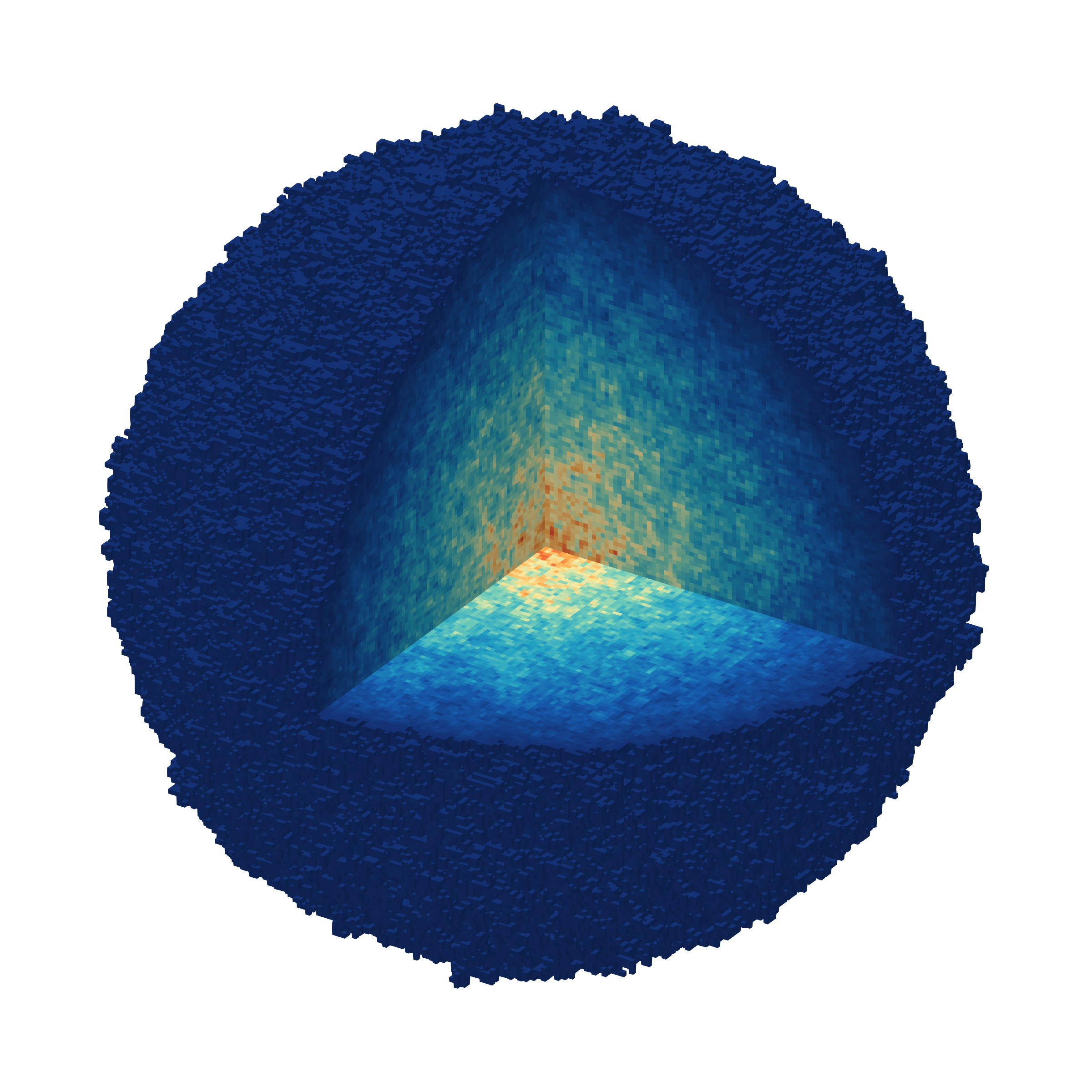}
\captionsetup{font=small}
\caption{Cutaway view of the range of the radially excited random walk in dimension $3$ after $10^8$ steps. Colours range from blue to red as the local time increases.}
\label{fig:range-local-times-d3}
\end{figure}

\section{Introduction}\label{sec:introduction}
The model of the excited random walk was introduced by Benjamini and Wilson \cite{BW} as a simple process with self-interaction: consider a walk on the lattice $\Z^d$, $d\geq 2$, with the modification that the first departure from a site is biased, while all subsequent departures are simple symmetric. In the original model, all the excitations point in the same fixed direction and it was proved that the walk is transient \cite{BW} and exhibits ballistic behaviour \cite{KozmaSpeed3,KozmaSpeed2,BR} in every dimension \(d\ge2\). Following the seminal work of Benjamini and Wilson \cite{BW}, a large literature has developed on more general excited random walks, notably with multiple excitations per site (\emph{a.k.a.} cookie random walks) and excitations in random environments \cite{Zerner2005,BS,KZ2008,MPRV,KP2017,KZSurvey}.

At an open problem session of the 2007 Oberwolfach workshop which the author attended, Kozma discussed a variant in which the first-visit bias depends on the position but always points towards the origin. Kozma later announced that such a walk is always recurrent in every dimension and also conjectured a limiting shape for the visited set, with radius of order \(n^{1/(d+1)}\) \cite{Kozma2007}. He subsequently communicated to the author a sketch of his unpublished recurrence argument (see also \cite{Kozma2013}). His approach is robust as it allows general inward-pointing cookies with varying biases but does not yield a shape theorem. In this paper, we use a different argument based on martingales constructed from Green functions to establish the shape conjecture for an excited random walk with an exactly radial mean drift of constant strength and obtain a uniform conical local-time profile from which recurrence of the walk follows.

Excited random walks belong to the broader class of self-interacting walks, whose non-Markovian dynamics and long memory make them notoriously difficult to study. Another prominent example in this family is the \emph{once-reinforced random walk}, where each edge receives extra weight after its first crossing (but not after subsequent crossings) \cite{Davis}. Its tendency to ``bump'' against the boundary of its trace is reminiscent of the excitation towards the centre considered here, and the two models are expected to behave similarly for strong reinforcements (but the radially excited random walk is more tractable because the inward drift is prescribed, whereas the bias is much more complicated in the reinforced case). Very recent works have brought substantial progress in the study of the once-reinforced random walk: Elboim and Kozma \cite{EK} proved transience for sufficiently weak reinforcement in dimensions $\geq 6$, while Bou-Rabee and Peres \cite{BP}, improving earlier estimates of Collevecchio and Tarr\`es \cite{CT}, obtained lower bounds of order \(n^{d/(d+1)}\) for the range in every dimension $d\ge 2$ and for every reinforcement strength (which is exactly the exponent obtained here for the radially excited random walk). Yet, recurrence in dimension $d=2$ and the conjectured phase transition for $d\geq 3$ from transience to recurrence (with a possible shape theorem) as reinforcement increases remain open.

\paragraph{AI disclaimer.}
The author painfully acknowledges the use of an LLM in the preparation
of this work but takes full responsibility for the mathematical correctness of the content and the quality of the presentation. The main framework of the proof (\emph{i.e.}, to consider Green martingales at each point of the space to simultaneously control the local times process) was the author's own years-old idea, but he could unfortunately never make it work \ldots The statement and proof of Proposition \ref{prop:deterministic-input} concerning the rigidity of the trace, which is the key to the shape theorem, were generated by GPT-6 in a depressingly short span of time while the author was having an enjoyable lunch with his colleagues at Orsay on a sunny day of September 2026.

\paragraph{Addendum.} A couple of days prior to the submission of this manuscript, Nguyen~\cite{Nguyen} proved a shape theorem for the centrally excited random walk, albeit with a different excitation formula that yields an $\ell^1$ ball (instead of the Euclidean $\ell^2 $ ball here) as the limiting shape. Yet, for the radial model considered in this paper, we obtain here a stronger conclusion, with almost sure convergence and an explicit uniform conical local-time profile. Furthermore, the proof presented here is simpler and much shorter than in \cite{Nguyen}, so the author believes the paper has its own merits and is worth presenting to the mathematical community.

\section{Model and main result}\label{sec:model}

We work here on $\Z^d$ with $d\geq 2$. Let \(e_1,\ldots,e_d\) denote the coordinate unit vectors of the space and let $\cE :=\{e_1,-e_1,\ldots,e_d,-e_d\}$. We write \(|x|\) for the Euclidean norm and, for \(x\in\R^d\), we define the radial unit vector
$$
 \bn(x):=\begin{cases}x/|x|,&x\ne0,\\0,&x=0.\end{cases}
$$
The parameter $\beta>0$ represents the strength of the radial drift and is fixed throughout the paper. The excitation (also called a ``cookie'' in the literature) provided at each site \(x\in\Z^d\) is given by a probability distribution $q_x$ on $\cE$ which is here assumed to satisfy the mean drift equality  
\begin{equation}\label{eq:radial-mean}
 \sum_{e\in\cE}q_x(e)e=-\beta\bn(x), \quad x \in \Z^d \setminus \{0\}.
\end{equation}
This means that the expected drift of the distribution is directed towards the origin and has strength $\beta$. For simplicity, we will assume that the distribution $q_0$ at the origin is symmetric. Note that, for $\beta \leq 1/d$, such a family of distributions may be realized by choosing 
\begin{equation}\label{eq:example}
 q_x(\pm e_i)=\frac1{2d}\mp\frac{\beta x_i}{2|x|},\qquad x = (x_1,\ldots,x_d) \neq 0.
\end{equation}
The radially excited random walk is a process $X = (X_n)_{n\geq 0}$ on $\Z^d$ defined by $X_0 = 0$ and the transition rule
\begin{equation}\label{eq:transition}
 \Pp(X_{n+1}=X_n+e \mid \F_n)
 =\begin{cases}
 q_{X_n}(e), & X_n\notin \{ X_i : i < n\}\\
 1/(2d),& \hbox{otherwise},
 \end{cases}
\end{equation}
where $\F_n := \sigma(X_0,\ldots,X_n)$ is the natural filtration of the process. In words, the walk jumps according to $q_x$ on its first visit to a site $x$ and then behaves like a simple symmetric random walk on subsequent visits.

\noindent The main result of this paper concerns the range of the walk at time $n$:
$$
A_n := \{X_0,\ldots,X_{n-1}\},
$$
and the local time at each site $x$ and time $n$:
$$
T_n(x) := \sum_{k=0}^{n-1} \one_{\{X_k=x\}}.
$$

\begin{theorem}[Shape theorem for radially excited random walk]\label{thm:main}
Let \(v_d\) denote the volume of the Euclidean unit ball and define
\begin{equation}\label{eq:constants}
 C_* :=\frac{2d\beta v_d}{d+1} \qquad\hbox{and}\qquad
 r_n :=(n/C_*)^{1/(d+1)}.
\end{equation}
It holds that, almost surely,
\begin{enumerate}[label=(\roman*)]
 \item For every \(\eps\in(0,1)\), 
\begin{equation}\label{eq:main-shape}
 D_{(1-\eps)r_n}\subset A_n\subset D_{(1+\eps)r_n}
 \qquad\hbox{for all $n$ large enough,}
 \end{equation} 
 where $D_r:=\{x\in\Z^d:|x|<r\}$. In particular, the volume and outer radius satisfy
 \begin{equation}\label{eq:main-clock}
 \sharp A_n \sim v_dr_n^d,\qquad
 \max_{x\in A_n}|x|\sim r_n.
 \end{equation}
 \item The local-time process is ``cone-shaped'' and the following uniform approximation holds:
 \begin{equation}\label{eq:main-local}
 \lim_{n\to\infty}\frac1{r_n}
 \sup_{x\in\Z^d}
 \left|T_n(x)-2d\beta(r_n-|x|)_+\right|=0.
 \end{equation}
 In particular, every fixed site is visited infinitely often, so the walk is recurrent and we have
\begin{equation}\label{eq:main-pointwise}
 \frac{T_n(x)}{r_n}\longrightarrow2d\beta
 \qquad\text{for every fixed }x\in\Z^d.
\end{equation}
\end{enumerate}
\end{theorem}

The proof of the theorem uses the following strategy:  we first introduce a family of martingales indexed by $\Z^d$ that link the local time of the walk at each site with the current range of the walk. The fluctuations of the martingales are shown to be negligible yielding an approximate deterministic relation between the local times and the range. A rigidity argument then shows that this relation forces a conical shape for the local-time process (Proposition~\ref{prop:deterministic-input}) but this does not specify the exact radius. An  annulus-crossing estimate is then used to rule out thin protrusions and identify the cone radius with the exit radius (Lemma~\ref{lem:annulus}). These results are first proved for the walk stopped at the exit times from large balls, and then transferred to deterministic times by the usual monotonicity argument.

\medskip

Throughout the paper, $C,c$ denote positive deterministic constants which may change from line to line but depend only on the fixed parameters $d$ and $\beta$, unless indicated otherwise.

\section{Proof of the main result}

\subsection*{Step 1. Relating the range and the local times through martingales}

Let $(S_k)$ denote the simple symmetric nearest-neighbour random walk on $\Z^d$ and let $\mathbf P_x$ denote its law started from $x$. We use the notation $G$ for the Green function of $S$ in dimension $d\ge 3$ and for its (negative) potential kernel in dimension $d=2$:
\begin{equation}
 G(x):= \begin{cases}
      \sum_{k=0}^\infty\mathbf P_0(S_k=x),&d\ge3,\\
      \lim_{N\to\infty}\sum_{k=0}^N
       \bigl[\mathbf P_0(S_k=x)-\mathbf P_0(S_k=0)\bigr],&d=2.
 \end{cases}
\end{equation}
The Green function is bounded in dimension $d\ge 3$ since the simple random walk is transient, and for $d=2$, the potential kernel diverges logarithmically. In particular, it is easy to check (see for instance \cite{LL}) that, in all dimensions $d\geq 2$, we have
\begin{align}
 |G(x)|&\le C\bigl(1+\log(1+|x|)\bigr),
 && x\in\Z^d,\label{eq:green-growth}\\
 |G(x+e)-G(x)|&\le C(1+|x|)^{1-d},
 && x\in\Z^d,\ e\in\cE.\label{eq:green-increment-bound}
\end{align}
The most important (somewhat defining) feature of the Green function is that it is harmonic everywhere except at the origin, where it has a unit defect: 
\begin{equation}\label{eq:green-equation}
 \frac1{2d}\sum_{e\in\cE} G(y+e) =  G(y) - \one_{\{y=0\}}.
\end{equation}
For $z\in \Z^d$, we denote by $G_z(x) := G(x-z)$ the shifted Green function corresponding to the walk $S$ started from $z$.

\medskip

\noindent For each $z\in\Z^d$, we define the process
\begin{equation}\label{eq:martingale-definition}
 M_n(z):=\sum_{k=0}^{n-1}\Big(
 G_z(X_{k+1})-G_z(X_k)
 -\E[G_z(X_{k+1})-G_z(X_k)\mid\F_k]\Big),\quad n\geq 0.
\end{equation}
By construction, $M(z)$ is an $\mathcal{F}_n$-martingale. It is the Doob martingale constructed from the process \(G_z(X_n)\) by compensating for the failure of harmonicity of $G_z$ at $z$ and the additional drift from the first departure from each site. Using the notation 
\begin{equation}\label{eq:cookie-correction}
 J(x,z) :=\sum_{e\in\cE}\left(q_x(e)-\frac1{2d}\right) \bigl[G(x+e-z)-G(x-z)\bigr],
\end{equation}
we can rewrite the previsible part as 
$$
 \E[G_z(X_{k+1})-G_z(X_k)\mid\F_k]=-\one_{\{X_k=z\}}+\one_{\{X_k\notin A_k\}}J(X_k,z).
$$
Thus, telescoping the increments of \(G_z(X_k)\) and counting each cookie correction once, we deduce from \eqref{eq:martingale-definition} the key identity:
 \begin{equation}\label{eq:green-identity}
 T_n(z)=G(-z)-G(X_n-z)+M_n(z)
                  +\sum_{x\in A_n}J(x,z).
\end{equation}

\medskip

\noindent\textbf{Heuristic for the shape theorem.} The identity \eqref{eq:green-identity} above is the founding stone to the proof of the shape theorem. Indeed, $G$ grows at most logarithmically and $M(z)$ is a ``nicely behaved'' martingale with bounded increments, so the contribution of the term $G(-z)-G(X_n-z)+M_n(z)$ in \eqref{eq:green-identity} is negligible compared to the scale of $T_n(z)$.  Therefore, we have, for every $z\in \Z^d$, 
\begin{equation}\label{eq:heuristic}
T_n(z)\approx\sum_{x\in A_n}J(x,z).
\end{equation}
The field $J$ is deterministic and $J(x,z)$ represents the influence of the excitation at site $x$ on visits of the walk to $z$: its strength decays with the distance and its sign is (asymptotically) positive when the excitation at $x$ points towards $z$ and negative when it points away. Relation \eqref{eq:heuristic} states that the local time is essentially determined by the integral of this field $J$ over its trace $A_n$. 

Furthermore, by definition, $A_n$ is exactly the set where the local times are non-zero, but we can also apply \eqref{eq:heuristic} for $z$ outside of the trace to find that 
\begin{equation}\label{eq:heuristic-vanishing}
\sum_{x\in A_n}J(x,z) \approx 0 \qquad \text{for every $z$ on the outer boundary of } A_n.
\end{equation}
The key observation is that this vanishing condition on the boundary is a very strong geometric constraint which, for the particular field $J$ considered here,  forces the trace to be approximately spherical. Then, once $A_n$ is known, \eqref{eq:heuristic} essentially determines the local times and forces the profile of $T_n$ to be asymptotically conical.

\medskip

\noindent\textbf{Recurrence of the walk in dimension $d=2$.} Recurrence of the excited walk in the general case follows from the shape theorem and the description of the local times. However, in the planar case, the martingale $M_n(0)$ can be used directly to provide an elementary proof of recurrence: because the excitation points toward the origin, we have $J(x,0)$ positive, except maybe for a finite number of sites $x$. Furthermore, in dimension $2$, the function $G$ diverges logarithmically, so on the event that the walk is transient, we have $G(X_n)\to-\infty$, forcing $M_n(0)\to-\infty$ (thanks to  \eqref{eq:green-identity}), which is impossible for a martingale with bounded increments.

\subsection*{Step 2. Control of the martingales at the exit times from a ball}

We formalize here the statement made in the heuristic that the martingales $M_n(z)$ are negligible compared to the scale of the local times. For $R>0$, recall that $D_R  := \{x\in\Z^d:|x|< R\}$
and define the exit time from the ball of radius \(R\) by
$$
 \tau_R:=\inf\{n\ge0:|X_n|\ge R\}.
$$
In particular, for $n \leq \tau_R$, we have $A_n\subset D_R$ and $|X_n|\le R+1$. The first result states that, with high probability, all the martingales associated with sites in $D_{2R}$ remain uniformly sublinear in $R$ while the walk remains in the ball $D_{R}$.

\begin{proposition}[Control of the martingales at exit times]\label{prop:noise}
Fix $R\geq 2$. We have $\tau_R<\infty$ a.s. Define the ``good event''
$$
 \mathcal G_R:=
 \left\{\max_{z\in D_{2R}}\max_{0\le n\le\tau_R}|M_n(z)|<R^{3/4}\right\}.
$$
There exist constants $C,C_{\rm time}$ such that
$$
 \Pp(\mathcal G_R^c)\le CR^{-3}.
$$
Furthermore, on $\mathcal G_R$, it holds that
\begin{equation}\label{eq:good-clock}
 \max_x T_{\tau_R}(x)\le CR
 \qquad \hbox{and}\qquad \tau_R\le C_{\rm time}R^{d+1}.
\end{equation}
\end{proposition}

\begin{proof}
The fact that $\tau_R < \infty$ a.s. is obvious since there is a single cookie per site and the process behaves like a simple random walk otherwise, so it cannot remain contained in any bounded region.

We first bound the two terms other than $M_n(z)$ in \eqref{eq:green-identity}. For $z\in D_{2R}$ and $n\le\tau_R$, estimate \eqref{eq:green-growth} gives the crude upper bound
\begin{equation}\label{eq:green-endpoint}
 |G(-z)-G(X_n-z)|\le C\log R
\end{equation}
for some $C>0$. Furthermore, \eqref{eq:green-increment-bound} and the definition of $J$ in \eqref{eq:cookie-correction} imply that $|J(x,z)|\le C(1+|x-z|)^{1-d}$.  Thus, for $z\in D_R$ and $n\le\tau_R$ we find that
$$
 \left|\sum_{x\in A_n}J(x,z)\right|
 \leq C\sum_{x\in D_R}(1+|x-z|)^{1-d}  \leq  C\sum_{x\in D_{2R}}(1+|x|)^{1-d} = O(R).
$$
Combining this bound with \eqref{eq:green-endpoint} in \eqref{eq:green-identity} gives
\begin{equation}\label{eq:pathwise-occupation}
 T_n(z)\le M_n(z)+CR,
 \qquad \hbox{for $z\in D_R$ and $n\le\tau_R$.}
\end{equation}
Now introduce
$$
 \gamma_R:=\inf\left\{n\ge0:\max_{z\in D_{2R}}|M_n(z)|\ge R^{3/4}\right\},
 \qquad \sigma:=\tau_R\wedge\gamma_R.
$$
Equation \eqref{eq:pathwise-occupation} gives
$$
\max_x T_\sigma(x)\le\max_x T_{\sigma-1}(x)+1
 \le R^{3/4}+CR+1\le C'R.
$$
For $z\in D_{2R}$, define the stopped martingale $N_n(z):=M_{n\wedge\sigma}(z)$. By \eqref{eq:green-increment-bound}, its increments satisfy
$$
 |N_{k+1}(z)-N_k(z)|
 \le a_k(z):=C(1+|X_k-z|)^{1-d}\one_{\{k<\sigma\}},
$$
where $a_k(z)$ is $\F_k$-measurable. Grouping the visits by site and using the bound on $T_\sigma$, we find that, for $z \in D_{2R}$, 
$$
 \sum_{k\ge0}a_k(z)^2 \le C\sum_{x\in D_R}T_\sigma(x)(1+|x-z|)^{2-2d}
 \le CR\sum_{y\in D_{3R}}(1+|y|)^{2-2d} \le CR\log R.
$$
Applying the Azuma--Hoeffding inequality with predictable increment bounds \cite[Corollary~1(a)]{HRMS} gives
$$
 \Pp\left(\sup_n|N_n(z)|\ge R^{3/4}\right)
 \le 2\exp\left\{-\frac{R^{3/2}}{CR\log R}\right\}
 \le 2e^{-c\sqrt R/\log R}.
$$
Finally, observe that on $\mathcal G_R^c$, at least one of these stopped martingales must reach the threshold, so a union bound over the $O(R^d)$ sites of $D_{2R}$ gives
$$
 \Pp(\mathcal G_R^c)\le CR^d e^{-c\sqrt R/\log R},
$$
which is $O(R^{-3})$ as claimed. Finally, on $\mathcal G_R$, \eqref{eq:pathwise-occupation} gives $\max_xT_{\tau_R}(x)\le CR$ and 
$$
\tau_R=\sum_{x\in D_R}T_{\tau_R}(x) \le CR \;\sharp D_R \le C_{\rm time} R^{d+1}.
$$
\end{proof}

\subsection*{Step 3. A conical local-time profile from the radial field}\label{sec:interface}

We now prove rigorously that the vanishing condition~\eqref{eq:heuristic-vanishing} on the boundary of the trace forces a conical shape for the local times. We first prove a deterministic rigidity statement and then apply it to the walk. The exact radius of the cone will be identified in Step 4.
We first define
$$
 K(y):=\begin{cases}y/|y|^d,&y\in\R^d\setminus\{0\},\\0,&y=0.\end{cases}
$$
For $R\ge2$ and a finite set $A\subset\Z^d$, we define the associated field
$$
 W_{R,A}(z):=\frac1{d v_dR}\sum_{x\in A}\bn(x)\cdot K(x-z),
$$
where $v_d$ is the volume of the Euclidean unit ball. The following proposition shows that if $W_{R,A}$ is small at some point outside $A$ closest to the origin, then it must be uniformly close to
a conical profile.

\begin{proposition}[Rigidity of the radial field]\label{prop:deterministic-input}
Let $R\ge2$ and fix $A\subset D_R$. Choose $z_0\notin A$ with minimal norm and set $r=|z_0|$. There is a constant $C$ depending only on $d$ such that
\begin{equation}\label{eq:field-rigidity}
 \sup_{z\in\Z^d}
 \left|W_{R,A}(z)-\left(\frac{r-|z|}{R}\right)_+\right|
 \le C\left(|W_{R,A}(z_0)|+R^{-1/d}\right)^{1/d}.
\end{equation}
\end{proposition}

\begin{proof}
Set $H=A\setminus D_{r}$. By the choice of $z_0$, we have $D_{r} \subset A$ so $A=D_{r}\sqcup H$. Let us for the time being admit that, for the ball $D_{r}$, it holds that
\begin{equation}\label{eq:ball-field}
 \sup_{z\in\Z^d}
 \left|W_{R,D_{r}}(z)-\left(\frac{r-|z|}{R}\right)_+\right|
 \le CR^{-1/d}.
\end{equation}
The proof of this estimate follows from a comparison with the continuous setting (where the identity is exact) and is postponed to the end of the proof.

\noindent We first upper bound the number of sites in $H$. For every $x\in H$, we have $|x|\ge|z_0|$; hence
$$
 x\cdot(x-z_0)
 =\frac{|x-z_0|^2+|x|^2-|z_0|^2}{2}
 \ge\frac{|x-z_0|^2}{2},
$$
and, since $|x-z_0|\le2|x|$ and $|x|\le R$, it follows that
$$
 \bn(x)\cdot K(x-z_0)
 \ge\frac{1}{2|x|\,|x-z_0|^{d-2}}
 \ge\frac{1}{(2R)^{d-1}}.
$$
This shows that all sites of $H$ contribute positively to the field at $z_0$, with a uniform lower bound. Consequently, we find that
$$
 W_{R,H}(z_0)\ge c\frac{\sharp H}{R^d}.
$$
At $z_0$, the cone in \eqref{eq:ball-field} vanishes. Hence \eqref{eq:ball-field} and the identity $W_{R,H}=W_{R,A}-W_{R,D_{r}}$ give
\begin{equation}\label{eq:tail-mass}
 \frac{\sharp H}{R^d}
 \le C\left(|W_{R,A}(z_0)|+R^{-1/d}\right).
\end{equation}
We now control the field $W_{R,H}$ uniformly in $z$. We use that $|\bn(x)\cdot K(x-z)|\le|x-z|^{1-d}$ for $x\ne z$ (and the diagonal term vanishes because $K(0)=0$). Let $m:=\sharp H>0$ and split the sum in $W_{R,H}$ between sites at distance at most  $m^{1/d}$ from $z$ and those farther away. For any $z\in\Z^d$, we find that
$$
\begin{aligned}
 d v_d R|W_{R,H}(z)|
 &\le \sum_{\substack{x\in H\\0<|x-z|\le m^{1/d}}}|x-z|^{1-d}
 +\sum_{\substack{x\in H\\|x-z|>m^{1/d}}}|x-z|^{1-d}\\
 &\le \sum_{0<|y|\le m^{1/d}}|y|^{1-d}
 +m\,m^{(1-d)/d}
 \le Cm^{1/d}.
\end{aligned}
$$
Finally, since $W_{R,A}=W_{R,D_{r}}+W_{R,H}$, combining this with \eqref{eq:tail-mass} and \eqref{eq:ball-field} proves \eqref{eq:field-rigidity}.

\medskip

It remains to prove \eqref{eq:ball-field}.
 Fix $z\in\Z^d$ and set $F_z(y)=\bn(y)\cdot K(y-z)$ for $y\in\R^d$. The divergence theorem gives
$$
 \int_{|y|=t}F_z(y)\dd S(y)
 =d v_d\,\one_{\{t > |z|\}},
 \qquad t>0,\quad t\ne|z|,
$$
where $\dd S(y)$ denotes surface measure on the sphere $\{|y|=t\}$. Indeed, $K(\cdot-z)$ has zero divergence away from $z$ and flux $d v_d$ around that point. Integrating over $t$ yields
\begin{equation}\label{eq:continuous-ball}
 \int_{|y|<r}F_z(y)\dd y=d v_d(r-|z|)_+.
\end{equation}
This integral is absolutely convergent because $|F_z(y)|\le|y-z|^{1-d}$. This identity explains the conical profile: the flux through a sphere is constant once its radius exceeds $|z|$.

We now compare this integral with the lattice sum using the unit cubes $Q_x=x+[-1/2,1/2)^d$.
Set aside the cubes meeting the sphere $\{|y|=r\}$ or either of the balls of radius $\sqrt R$ centered at $0$ and $z$.
Since $r\le R+1$ and $d\ge2$, their number, and hence their total volume, is at most $C R^{d-1}$.
Their total contribution to the sum and the integral is therefore at most $CR^{1-1/d}$, by the bound for finite sums above and its integral analogue
$$
 \int_E|y-z|^{1-d}\dd y\le C\vol(E)^{1/d},
$$
obtained by splitting at distance $\vol(E)^{1/d}$ from $z$. Every remaining cube lies entirely inside or outside the ball. On those inside, both singularities are at distance at least $\sqrt R$, so the mean value theorem gives
$$
 |F_z(x)-F_z(y)|\le CR^{-1/2}|x-z|^{1-d},\qquad \hbox{for } y\in Q_x.
$$
Summing over these cubes costs at most $C\sqrt R$, since the bound for finite sums above gives $\sum_{x\in D_r\setminus\{z\}}|x-z|^{1-d}\le CR$.
Consequently, uniformly in $z$,
$$
 \left|\sum_{x\in D_r}F_z(x)-\int_{|y|<r}F_z(y)\dd y\right|
 \le C(R^{1-1/d}+\sqrt R)\le CR^{1-1/d}.
$$
Together with \eqref{eq:continuous-ball}, division by $d v_dR$ proves \eqref{eq:ball-field}.

\end{proof}

\medskip

\noindent\emph{Application to the walk.}
We want to use Proposition \ref{prop:deterministic-input} to prove that the local time process has a conical shape at time $\tau_R$. To do so, we first relate $W_{R,A_{\tau_R}}$ to the local times. The Green asymptotics in \cite[Theorems~4.3.1 and~4.4.4]{LL} give
$$
 G(y+e)-G(y)
 =-\frac2{v_d}K(y)\cdot e+O(|y|^{-d}),
 \qquad y\in\Z^d\setminus\{0\},\ e\in\cE.
$$
Using \eqref{eq:radial-mean} and \eqref{eq:cookie-correction}, we find that
\begin{equation}\label{eq:J-expansion}
 J(x,z)=\frac{2\beta}{v_d}\bn(x)\cdot K(x-z)
             +O(|x-z|^{-d}),\qquad x\ne z.
\end{equation}
Summing the remainders in \eqref{eq:J-expansion} costs at most $C\log R$ for $z\in D_{2R}$, since
$$
 \sum_{x\in D_R\setminus\{z\}}|x-z|^{-d}\le C\log R.
$$
Together with \eqref{eq:green-identity}, \eqref{eq:green-endpoint} and the martingale bound from Proposition~\ref{prop:noise}, we find that
\begin{equation}\label{eq:discrete-approximation}
 \max_{z\in D_{2R}}
 \left|\frac{T_{\tau_R}(z)}{2d\beta R}- W_{R,A_{\tau_R}}(z)\right|
 \le CR^{-1/4} \quad\text{on }\mathcal G_R.
\end{equation}
This is the quantitative version of \eqref{eq:heuristic}.

Let us now choose $z_0\notin A_{\tau_R}$ with minimal norm and set
$$
a_R:=|z_0|/R.
$$
Thus, $R a_R$ is the radius of the largest filled open ball centered  at the origin. Because $z_0$ is outside the range, we have $T_{\tau_R}(z_0)=0$; hence \eqref{eq:discrete-approximation} gives $|W_{R,A_{\tau_R}}(z_0)|\le CR^{-1/4}$. Then, applying Proposition~\ref{prop:deterministic-input} to the set $A_{\tau_R}$ and using \eqref{eq:discrete-approximation} again yields
\begin{equation}\label{eq:cone-power}
 \sup_{z\in\Z^d}
 \left|\frac{T_{\tau_R}(z)}R
       -2d\beta\left(a_R-\frac{|z|}{R}\right)_+\right|
 \le CR^{-\gamma}\quad\text{on }\mathcal G_R,
\end{equation}
with $\gamma=\min\{1/(4d),1/d^2\}$. Note that we can run the supremum over all $\Z^d$ because both the local time and the cone vanish outside $D_{2R}$.

Thus, we have now proved that, with high probability, the local time profile of the walk at time $\tau_R$ of the exit of $D_R$ has a conical shape with radius (possibly random) $R a_R$. 

\subsection*{Step 4. Identification of the exact radius}\label{sec:annulus}

We now show that the asymptotic radius of the cone is actually $R$, \emph{i.e.} $a_R$ converges to $1$. The argument is probabilistic: if this were not the case, then the walk would have to cross an annulus of width $(1-a_R)R$ with small local time but this can be excluded using (again) a martingale argument.

\begin{lemma}[Crossing an annulus with small local times]\label{lem:annulus}
There exist constants $c,C>0$ such that, for all $1/2<\eta<1$, $R\ge2$ and positive integers $L,M$, we have
\begin{equation}\label{eq:annulus-bound}
 \Pp\left(\tau_R\le M,\ \sup_{|x|>\eta R}T_{\tau_R}(x)<L\right)
 \le CM\exp\left\{-\frac{c(1-\eta)R}{L}\right\}.
\end{equation}
\end{lemma}

\begin{proof}
Choose $c>0$ sufficiently small (whose exact value will be specified later). It suffices to consider $L\le cR$, since otherwise the bound is trivial after increasing $C$. In particular, we may assume $R\ge4$.
Set $\theta:=c/L$ and define $f(x):=e^{\theta|x|}$. Just as in Step 1, we consider again a Doob martingale but associated this time with $f(X_n)$ and previsible increments 
$$
 \mu_k:=\E[f(X_{k+1})-f(X_k)\mid\F_k].
$$
We want to control $\mu_k$. To this end, assume $|x|>R/2$ and $e\in\cE$; then Taylor's formula gives
$$
f(x+e)-f(x) \le\theta f(x)\left[\bn(x)\cdot e+C_0\left(R^{-1}+\theta\right)\right],
$$
where the constant $C_0$ depends only on $d$, and the bound is uniform in $R,\theta,x,e$. By assumption \eqref{eq:radial-mean}, the radial component of the next step has conditional mean $-\beta$ at the first departure from $x$ and zero afterward. Thus, taking the conditional expectations yields
$$
 \mu_k\le\theta f(X_k)\left[C_0\left(R^{-1}+\theta\right)
 -\beta\one_{\{X_k\notin A_k\}}\right]\quad\hbox{when $|X_k|>\eta R$.}
$$
Consider now the stopping time
$$
 \sigma:=\tau_R\wedge
 \inf\left\{n:\sup_{|x|>\eta R}T_n(x)\ge L\right\} \wedge M.
$$
Now, each site contributes exactly one cookie. Thus, using that $T_\sigma(x)\le L$, we find that, for any site $x$ with $|x|>\eta R$ visited, its contribution is bounded by 
$$
 \begin{aligned}
 \sum_{\substack{k<\sigma\\X_k=x}}\mu_k
 &\le\theta f(x)\left[C_0\left(R^{-1}+\theta\right)T_\sigma(x)-\beta\right]\\
 &\le\theta f(x)\left[C_0\left(R^{-1}+\theta\right)L - \beta\right]
 \le\theta f(x)(2C_0c-\beta) \le 0,
 \end{aligned}
$$
where we used $L/R\le c$ and $L\theta=c$.  The last inequality holds as soon as $c\le\min\{1/4,\beta/(2C_0)\}$ (which depends only on $d$ and $\beta$) so we can now fix $c$ equal to this minimum.

For the time indices such that $|X_k| \leq \eta R$, we just use the crude bound $\mu_k\le e^{\theta(\eta R+1)}$.
Putting things together and applying the optional stopping theorem to the martingale $f(X_n)-\sum_{k<n}\mu_k$ with the (bounded by $M$) stopping time $\sigma$, we find that
$$
 \E(f(X_\sigma))=1+\E\big(\sum_{k<\sigma}\mu_k\big)
 \le1+M e^{\theta(\eta R+1)}.
$$
On the event in \eqref{eq:annulus-bound}, we have $\sigma=\tau_R$ and $f(X_\sigma)\ge e^{\theta R}$. Consequently,
$$
 \Pp\left(\tau_R\le M,\ \sup_{|x|>\eta R}T_{\tau_R}(x)<L\right) \le e^{-\theta R}\bigl(1+M e^{\theta(\eta R+1)}\bigr) \le CM e^{-c(1-\eta)R/L},
$$
as claimed.
\end{proof}

We can now identify the radius of the cone. Let $\eps\in(0,1)$ and put $\eta=1-\eps/2$. On the event $\mathcal G_R\cap\{a_R\le1-\eps\}$, the conical approximation \eqref{eq:cone-power} implies
$$
 \sup_{|z|>\eta R}T_{\tau_R}(z)\le C_{\rm cone}R^{1-\gamma},
$$
while Proposition~\ref{prop:noise} gives $\tau_R\le C_{\rm time}R^{d+1}$.
Apply Lemma~\ref{lem:annulus} with
$$
 L_R:=\left\lceil C_{\rm cone}R^{1-\gamma}\right\rceil+1,
 \qquad M_R:=\left\lceil C_{\rm time}R^{d+1}\right\rceil.
$$
For all sufficiently large $R$, we obtain
$$
 \Pp(a_R\le1-\eps) \leq \Pp(\mathcal{G}_R^c) + \Pp\left(\tau_R\le M_R,\ \sup_{|x|>\eta R}T_{\tau_R}(x)<L_R\right)
 \le CR^{-3}+CR^{d+1}e^{-c\eps R^\gamma},
$$
where $c,C>0$ depend only on $d$ and $\beta$. The right side is summable over integer $R$. Therefore, the 
Borel--Cantelli lemma ensures that $\liminf_{R\to\infty} a_R\ge1-\eps$ a.s. and therefore $\lim_{R\to\infty} a_R = 1$ a.s. The events $\mathcal G_R$ also hold eventually almost surely thanks to the Borel--Cantelli lemma. Finally, using  that $(a-u)_+$ is Lipschitz in $a$, we conclude from \eqref{eq:cone-power} that 
\begin{equation}\label{eq:exit-profile-final}
 \lim_{R\to\infty}\frac{1}{R}
 \sup_{z\in\Z^d}
 \left|T_{\tau_R}(z)-2d\beta(R-|z|)_+\right|=0
 \qquad\text{a.s.}
\end{equation}
We point out that the limit above is \emph{a priori} valid along integer $R$ but trivially extends to all real values by monotonicity of the exit times and local times w.r.t. $R$.

\subsection*{Step 5. From exit times to deterministic times}\label{sec:completion}

\begin{proof}[Completion of the proof of Theorem~\ref{thm:main}]
We first identify the exact rate of growth of $\tau_R$. By definition, $\tau_R$ is the sum of the local times over all sites in $D_R$, which has $O(R^d)$ sites, so using \eqref{eq:exit-profile-final} gives, a.s. 
$$
 \tau_R=\sum_{z\in D_R} T_{\tau_R}(z) = 2d \beta\sum_z(R-|z|)_++o(R^{d+1}).
$$
A Riemann-sum approximation yields
$$
 \frac1{R^{d+1}}\sum_z(R-|z|)_+
 \longrightarrow\int_{\R^d}(1-|y|)_+\dd y
 =d v_d\int_0^1(1-r)r^{d-1}\dd r
 =\frac{v_d}{d+1}.
$$
Therefore, we conclude that
\begin{equation}\label{eq:clock-final}
 \lim_{R\to\infty} \frac{\tau_R}{R^{d+1}} =  \frac{2d\beta v_d}{d+1} =: C_* \quad\hbox{a.s.}
\end{equation}
We now transfer the result from hitting times to deterministic times using a standard ``sandwich argument'' based on monotonicity of the local times. For $n\ge1$, define $j_n := \sup(j\in \N :\, \tau_{j}\le n) $. We have
\begin{equation}\label{eq:time-squeeze}
 \tau_{j_n}\le n<\tau_{j_n+1}.
\end{equation}
On  the one hand, using \eqref{eq:clock-final} and recalling that $r_n :=(n/C_*)^{1/(d+1)}$, it follows that 
$$
\frac{j_n}{r_n}\longrightarrow 1.
$$
On the other hand, monotonicity of local times combined with \eqref{eq:time-squeeze} implies, simultaneously for every $z$, that
$$
 T_{\tau_{j_n}}(z)\le T_n(z)\le T_{\tau_{j_n+1}}(z).
$$
Since $j_n/r_n\to 1$ and $r\mapsto(r-|z|)_+$ is Lipschitz, \eqref{eq:exit-profile-final} shows that both the lower and upper bounds are equal to $2d\beta(r_n-|z|)_++o(r_n)$, uniformly in $z$. The same therefore holds for $T_n(z)$, proving 

\begin{equation}\label{eq:deterministic-profile}
\lim_{n\to\infty}\frac1{r_n}
\sup_{z\in\Z^d}\left|T_n(z)-2d\beta(r_n-|z|)_+\right|=0
\qquad\text{a.s.},
\end{equation}
which is exactly statement \eqref{eq:main-local} of Theorem \ref{thm:main}. In particular, for a fixed site $z$, this gives 
$$
 \frac{T_n(z)}{r_n}
 \longrightarrow2d\beta>0,
$$
which is the limit stated in \eqref{eq:main-pointwise} of Theorem \ref{thm:main} and implies the recurrence of the walk.

We now turn to the range $A_n$. Fix $\eps\in(0,1)$. The approximation \eqref{eq:deterministic-profile} just established gives
$$
 \min_{z\in D_{(1-\eps)r_n}}T_n(z)\ge2d\beta\eps r_n-o(r_n)>0
$$
for all sufficiently large $n$. Thus $D_{(1-\eps)r_n}\subset A_n$. On the other hand, $n<\tau_{j_n+1}$ implies $A_n\subset D_{j_n+1} \subset D_{(1+\eps)r_n}$ for $n$ large since $(j_n+1)/r_n\to 1$. Thus, we have proved that 
$$
D_{(1-\eps)r_n}\subset A_n \subset D_{(1+\eps)r_n} \quad \hbox{for all $n$ large enough a.s.}
$$
Finally, the asymptotics $\sharp A_n \sim v_dr_n^d$ and $\max_{x\in A_n}|x|\sim r_n$ follow trivially from this shape theorem  and the estimate $\sharp D_r\sim v_dr^d$ so the proof of the theorem is complete.
\end{proof}

\end{document}